\documentclass[12pt]{amsart}

\usepackage{amsmath}

\calclayout

\usepackage{amsfonts,amssymb,amsthm,enumitem}
\usepackage{mathtools}
\usepackage{nicefrac,setspace}
\usepackage{xcolor}

\usepackage{hyperref}       % hyperlinks
\usepackage{url}            % simple URL typesetting
\usepackage{booktabs}       % professional-quality tables
\usepackage{amsfonts}       % blackboard math symbols
\usepackage{nicefrac}       % compact symbols for 1/2, etc.
\usepackage{microtype}      % microtypography

\newtheorem{theorem}{Theorem}
\newtheorem{claim}[theorem]{Claim}
\newtheorem*{theorem*}{Theorem}
\newtheorem*{claim*}{Claim}
\newtheorem*{remark}{Remark}
\newtheorem*{lemma*}{Lemma}
\newtheorem*{defn}{Definition}
\newtheorem{lemma}[theorem]{Lemma}

\newtheorem*{corollary*}{Corollary}

\newtheorem{prop}[theorem]{Proposition}

\renewcommand{\P}{\mathop {\mathbb P}}

\newcommand{\eps}{\epsilon}

\newcommand{\R}{\mathbb{R}}

\newcommand{\E}{\mathop \mathbb{E}}

\newcommand{\ip}[2]{\left\langle #1,#2\right\rangle}

\title{Slicing the hypercube is not easy}
\author{Gal Yehuda}
\address{Department of Computer Science, Technion-IIT}
\email{ygal@technion.ac.il}
\author{Amir Yehudayoff}
\address{Department of Mathematics, Technion-IIT}
\email{yehuday@technion.ac.il}

\begin{document}

\begin{abstract}
We prove that at least $\Omega(n^{0.51})$ hyperplanes
are needed to slice all edges of the $n$-dimensional hypercube.
We provide a couple of applications:
lower bounds on the computational complexity of parity,
and a lower bound on the cover number of the hypercube
by skew hyperplanes.
\end{abstract}

\maketitle

\section{Introduction}

The Boolean hypercube has a natural geometric representation in Euclidean space.
Its vertices are the $2^n$ points $\{\pm 1\}^n \subset \R^n$.
Its edges are the $n 2^{n-1}$ line segments 
$[x,y] \subset \R^n$ connecting
every two adjacent vertices  
(i.e., $x$ and $y$ differ in a single coordinate).
This representation is important and useful.
It leads to many applications
in algorithm design, optimization, machine learning, and more.

The question we address is
{\em How many hyperplanes are needed to slice all edges?}
A hyperplane $h = \{z \in \R^n : \ip{z}{v} = \mu\}$ slices
the edge $[x,y]$ if $x$ and $y$ lie on two different sides of $h$; that is, $(\ip{x}{v}-\mu)(\ip{y}{v}-\mu) < 0$.

This question has attracted a lot of attention.
It was first studied in the 1970s by O'Neil
with motivation from machine learning~\cite{O_Neil_1971},
and by Gr\"unbaum with motivations from convex geometry~\cite{10.2307/2317670}.
Additional motivation comes from computational complexity theory (see e.g.~\cite{paturi1990threshold}).
For an excellent introduction to this and related topics,
see the survey of Saks~\cite{Saks_1993}.

It is tempting to guess that $n$ hyperplanes are needed to slice all edges. 
This, however, was refuted by M.\ Paterson in an unpublished
example (see~\cite{Saks_1993}).
Paterson showed that only $5$ hyperplanes suffice to
slice the $6$-dimensional cube. By sub-additivity,
$\lceil \tfrac{5n}{6} \rceil$ hyperplanes suffice to slice the $n$-dimensional cube.
This is the best upper bound known even today.

In terms of lower bounds, O'Neil proved that the minimum number of slicing hyperplanes 
is $\geq \Omega(n^{0.5})$. The reason is that a single hyperplane
can slice a fraction of $\leq O(n^{-0.5})$ edges,
and this is sharp~\cite{baker1969generalization}.
In several special cases, better lower bounds are known.
In dimension $n\leq 4$, Emamy-Khansary~\cite{EMAMYKHANSARY1986221} proved
that at least $n$ hyperplanes are needed.
If all entries in all the normal vectors $v_i$ are non-negative
then at least $n$ hyperplanes are needed~\cite{AHLSWEDE1990137}.
If all normal vectors have coefficients in $\{\pm 1\}$
then at least $\frac{n}{2}$ hyperplanes are needed~\cite{138716d202994622afcef5b40279ef04,Saks_1993}.

We provide the first improvement\footnote{We did not attempt
to optimize the constant in the exponent.} in fifty years.

\begin{theorem*}
Slicing all edges in $\{\pm 1\}^n$ requires 
$\Omega(n^{0.51})$ hyperplanes.
\end{theorem*}

Before discussing the proof,
we provide a couple of applications.

\subsection{Applications}
The theorem immediately implies new lower
bounds in computational complexity theory;
specifically, for computing parity with threshold circuits.
The high-level reason is that the $k$ gates in the first (closest
to the inputs) layer of a threshold
circuit for parity always yield $k$ slicing hyperplanes.
The best known lower bounds were proved 
about thirty years ago.
Paturi and Saks proved an $\Omega(n^{1.5})$
lower bound on the number of wires in any depth-two
threshold circuit computing parity~\cite{paturi1990threshold}.
Impagliazzo, Paturi, and Saks
proved a super-linear lower bound on the number
of wires in any constant-depth threshold circuit for parity~\cite{impagliazzo1997size}.
In special cases, better lower bounds are known.
The main result of Paturi and Saks is, in fact, that if the coefficients
in the depth-two circuit are integers with bit-complexity $O(\log(n))$,
then a stronger $\Omega(n^2/\log^2(n))$
lower bound on the wire complexity holds~\cite{paturi1990threshold}.

The theorem implies the first $\omega(n^{0.5})$ lower bound on the number
of gates in the first layer of any threshold circuit computing parity, no matter how deep it is.
It also implies the first $\omega(n^{1.5})$
lower bound on the number of wires in any depth-two
threshold circuit for parity.
An extension to other constant depths can be obtained using the ideas
in~\cite{impagliazzo1997size}.

\begin{corollary*}
The number of wires in a depth-two threshold circuit
for parity is $\Omega(n^{1.51})$.
\end{corollary*}

\begin{proof}
Use the theorem and Lemma 1 in~\cite{paturi1990threshold}.
\end{proof}

The second application is more geometric.
Instead of slicing edges, we can ask to cover vertices.
A hyperplane $\{z \in \R^n : \ip{z}{v} = \mu\}$ covers the
vertex $x$ if $\ip{x}{v}=\mu$.
{How many hyperplanes are needed to cover all vertices?}
Just two.
A natural and standard way to make the question
more interesting is to ask the hyperplanes to be skewed
(i.e., all entries in all normal vectors are nonzero).

For the slicing problem, the skew restriction is not severe. 
A small perturbation of a slicing family is also a slicing family.
So, we can assume that all the hyperplanes are skewed.

For the covering problem, it creates a huge difference.
The minimum number of skewed hyperplanes needed to cover
all vertices is known to be at least $\Omega(n^{0.5})$; follows from~\cite{littlewood_offord_1939,erdos1945lemma}.
Again, better results are known in special cases (see~\cite{Saks_1993,138716d202994622afcef5b40279ef04}).

We prove the first improvement on skew cover numbers as well.

\begin{corollary*}
Covering all vertices in $\{\pm 1\}^n$  requires 
$\Omega(n^{0.51})$ skew hyperplanes.
\end{corollary*}
%Also assume that
%all entries of $v_i$ are integers
%(this should hold w.l.o.g.; see remark in paper with Anup
%+ add $D$-ary representation of integers as well).

\begin{proof}
Let $v_1,\ldots,v_k \in \R^n$ be skewed and let $\mu_1,\ldots,\mu_k \in \R$.
If the $k$ hyperplanes defined by $v_i$ and $\mu_i$
are covering, then duplicating each $v_i$ with two thresholds $\mu_i \pm \beta$ for some small $\beta>0$ gives $2k$
hyperplanes that slice all edges.
\end{proof}

\subsection{Outline}
The proof of our main result
brings together ideas from several fields (geometry,
probability theory, partially ordered sets, linear algebra, 
and some specialized structural arguments).
We now sketch some of the main ideas.

Let us assume that $k \leq n^{0.51}$ hyperplanes slice all edges of the cube.
Our task is to locate the missing edge, an edge that is not sliced.
Where should we look for it?
We can try to use randomness.
This is indeed helpful and proves the $\Omega(n^{0.5})$
lower bound, but it is not clear how to get a better lower bound
using just randomness.
We can try to use algebra. This is also fruitful in some cases~\cite{138716d202994622afcef5b40279ef04,Saks_1993}.
But we are dealing with general real numbers
and with the notion of slicing, so it is not clear how algebra can help.
We can try to use topology or geometry.

Our opening move identifies a high-level connection
between the slicing problem and Tarski's plank problem~\cite{tarski1932uwagi}.
The plank problem asks about the minimum total width
of planks (regions between parallel hyperplanes)
needed to cover a convex body.
It was solved by Bang in his celebrated work~\cite{10.2307/2031721}.
We use a lemma that Ball isolated from Bang's solution~\cite{Ball_1991}.

\begin{lemma}[Bang]
\label{lem:bang}
Let $M$ be a $k \times k$ real symmetric matrix with ones on the diagonal.
Let $\gamma_1,\ldots,\gamma_k \in \R$ and $\theta \geq 0$.
Then, there is $\eps \in \{\pm 1\}^k$ so that
for every $i \in [k]$,
$$\big| \theta (M\eps)_i  - \gamma_i \big| \geq \theta.$$
\end{lemma}

The proof of Bang's lemma is elegant and surprisingly short.
The lemma is relevant to the slicing problem because
it states the {\em existence} of a vector.
This vector is the starting point of our search.

But before we are able to use Bang's lemma
we need to understand where to focus our attention.
We identify a useful partition of the $k$ normal vectors
to two parts with different behaviors (Section~\ref{sec:decom}).
To utilize this partition,
we develop a couple of general anti-concentration results.
One strong anti-concentration result for vectors with many scales
(Section~\ref{sec:scales}).
The second set of results is general and 
extends several well-known theorems.

An antichain in $\{0,1\}^n$ 
is a family of sets with no pairwise strict inclusions.
Sperner's lemma provides a sharp upper bound on the size of antichains~\cite{sperner1928satz}.
The lemma states that if $A \subset \{0,1\}^n$ is an antichain
then 
$$|A| \leq \max_{\ell \in \{0,1,\ldots,n\}} {n \choose \ell}.$$
It is fundamental and has many applications 
(specifically, in the context of anti-concentration; see~\cite{erdos1945lemma}).
There are several extensions of Sperner's lemma
to various settings.
We need a version of the
lemma for product measures.

Let $P$ be a product distribution on $\{0,1\}^n$.
We say that $P$ is non trivial if $p_j(1-p_j) > 0$ for all $j \in [n]$,
where $p_j = \Pr_{z \sim P}[z_j = 1]$.
Aizenman et al.~\cite{aizenman2009bernoulli} proved 
a version of Sperner's lemma for product distributions.
They showed that the maximum measure of an antichain is
$\leq O(\tfrac{1}{\alpha \sqrt{n}})$ where
$\alpha = \min \{ p_1 ,\ldots,p_n , 1-p_1,\ldots,1-p_n\}$.
Their proof uses Bernoulli decompositions of random variables.
Their bound, however, is not strong enough for our purposes.
We obtain the sharpest possible bound,
and develop some theory in the process (see Section~\ref{sec:antiVer}).

\begin{theorem}
\label{thm:sperner}
For every non trivial product distribution $P$
on $\{0,1\}^n$, and for every antichain $A \subset \{0,1\}^n$,
\begin{align}
\label{eqn:Sper}\Pr_{z \sim P}[ z \in A] \leq \max_{\ell \in \{0,1,\ldots,n\}}
\Pr_{z \sim P}[|z|=\ell].
\end{align}
\end{theorem}

The Lubell–Yamamoto–Meshalkin inequality is stronger than
Sperner's lemma~\cite{yamamoto1954logarithmic,meshalkin1963generalization,lubell1966short}. It states that if $A \subset \{0,1\}^n$ is an antichain then
$$\sum_{\ell =0}^n \frac{|A_\ell|}{{n \choose \ell}} \leq 1,$$
where $A_\ell = \{ a \in A : |a|=\ell\}$.
The LYM inequality was also generalized in several ways (see e.g.~\cite{erdHos1992sharpening}),
and is sometimes more useful than Sperner's lemma (see e.g.~\cite{aizenman2009bernoulli}).
The following lemma is a strict generalization of the LYM inequality,
and puts the inequality in context.

\begin{theorem}
\label{thm:LYM}
For every non trivial product distribution $P$
on $\{0,1\}^n$, and for every antichain $A \subset \{0,1\}^n$,
$$\sum_{\ell=0}^n 
\Pr_{z \sim P}[z \in A||z|=\ell]\leq 1.$$
\end{theorem}

What underlying property of product measures allows to control
the measure of antichains?
We identify the following mechanism.
There is a way to sample a full chain
in a way that respects the measure.

\begin{lemma}
\label{lem:distChain}
Let $P$ be a non trivial product distribution $P$
on $\{0,1\}^n$.
For $\ell \in \{0,1,\ldots,n\}$,
let $P_\ell$ be the distribution of $z \sim P$ conditioned on the event that
$|z| = \ell$.
Then, there is a distribution on maximal chains
$$\emptyset = c_0 \subset c_1 \subset \ldots \subset c_n = [n]$$
so that for every $\ell \in \{0,1,\ldots,n\}$,
the set $c_\ell$ is distributed according to $P_\ell$.
\end{lemma}

To make use of Theorem~\ref{thm:sperner},
we need to control the right hand side of~\eqref{eqn:Sper}.
Namely, we need to prove an anti-concentration result
for general product measures.
This is done in Section~\ref{sec:prodMeasure}.

\begin{theorem}
\label{thm:anticon}
There is a constant $C_1>0$ so that the following holds.
Let $P$ be a product distribution on $\{0,1\}^n$.
Denote the variance of $|z|$ for $z \sim P$ by
$\sigma^2_P = \sum_j p_j(1-p_j)$.
Then, 
$$\max_{\ell \in \{0,1,\ldots,n\}} \Pr_{z \sim P} [ |z|= \ell ] \leq \frac{C_1}{\sigma_P}.$$
\end{theorem}

What we really care about is not antichains of vertices
but antichains of edges,
because we are interested in the edges that are sliced by a hyperplane.
Edge antichains are somehow more complicated than vertex antichains.
A specific difficulty we need to deal with is orientation.
In the hypercube with the uniform distribution,
all directions ``look the same''.
For general product distributions, this is no longer true.
There are special directions.
This becomes problematic now
because by choosing an edge we must commit to some direction.
The chosen direction might not agree with the (unknown) orientation.
All this is explained in Section~\ref{sec:antiedges}.

After discovering all these general results,
we identified a second approach for proving what we need
about antichains.
We can replace {\em general} oriented antichains of edges
by antichains of edges of a specific type.
Antichains that are generated
by oriented monotone Boolean functions.
This naturally leads us to 
Fourier analysis over the hypercube; see e.g.~\cite{kahn1989influence,bshouty1996fourier,o2014analysis}
and references within.
For the details, see Section~\ref{sec:monotone}.

\begin{remark}
To summarize,
there are two avenues for proving the ``antichain part'' of our main result.
One wide avenue that extends Sperner's lemma
and the LYM inequality (Sections~\ref{sec:antiVer}, \ref{sec:prodMeasure}
and~\ref{sec:antiedges}).
The second avenue is more specific, quicker and relies on Fourier analysis
(Section~\ref{sec:monotone}).
\end{remark}

\section{Antichains of vertices}
\label{sec:antiVer}

Here we extend Sperner's lemma and the LYM inequality to
general product measures.
The main proposition is natural,
but the proof we found is technical.
Recall that $P_\ell$ denotes the distribution of
$z \sim P$ conditioned on $|z|=\ell$.

\begin{prop}
\label{prop:Probconst}
Let $P$ be a non trivial product distribution on $\{0,1\}^n$.
Let $\ell \in \{0,1,\ldots,n-1\}$.
Then, there is a probability distribution on pairs $
(c_\ell,c_{\ell+1}) \in {[n] \choose \ell} \times 
{[n] \choose \ell+1}$ so that the following hold:
\begin{enumerate}
\item $c_\ell \subset c_{\ell+1}$ almost surely.
\item $c_\ell$ is distributed like $P_\ell$.
\item $c_{\ell+1}$ is distributed like $P_{\ell+1}$.
\end{enumerate}

\end{prop}

\begin{proof}
Choose $c_\ell$ according to $P_\ell$.
For $s \subset [n]$ of size $|s|=\ell$ and $j \not \in s$,
we need to decide what is the probability of $c_{\ell+1} = s \cup \{j\}$ conditioned on $c_\ell = s$.

Let $q \in \R^n$ be so that for all $j \in [n]$,
$$q_j = \frac{p_j}{1-p_j}>0.$$
Consider the symmetric polynomial
$$g_\ell(q) = \sum_{t \subseteq [n] : |t|=\ell} \prod_{j \in t} q_j ,$$
where $g_0(q)=1$.
For $s \subset [n]$ of size $\ell$,
\begin{align*}
\Pr[z=s] 
& = \prod_{a \in s} p_a \cdot \prod_{j \not \in s} (1-p_j) 
 = \prod_{j \in [n]} (1-p_j) \cdot \prod_{a \in s} q_a .
\end{align*}
It follows that
\begin{align*}
P_\ell(s) = \Pr_{z \sim P}[z=s||z|=\ell]
& = \frac{\prod_{a \in s} q_a}{g_\ell(q)}.
\end{align*}
Let  
$$h_{s,j}= \sum_{t \subset [n] : |t|=\ell, j \not \in t} \frac{1}{|(s \cup \{j\}) \setminus t|} \prod_{a \in t} q_a >0.$$
When $\ell=0$ set $h_{s,j} = 1$.
For every $s$,
\begin{align*}
\sum_{j \not \in s} q_j h_{s,j}
& = \sum_{j \not \in s} q_j\sum_{t: |t|=\ell, j \not \in t} \frac{1}{|(s \cup \{j\}) \setminus t|} \prod_{a \in t} q_a \\
& = \sum_{t: |t|=\ell} \sum_{j \not \in s \cup t}  \frac{1}{|(s \cup \{j\}) \setminus t|} \prod_{a \in t \cup \{j\}} q_a \\
& = \sum_{t: |t|=\ell} \sum_{j \not \in s \cup t}  \frac{1}{1+|s \setminus t|} \prod_{a \in t \cup \{j\}} q_a \\
& = \sum_{t': |t'|=\ell+1} \prod_{a \in t'} q_a \sum_{j \in t' , j \not \in s}  \frac{1}{1+|s \setminus t'|}  \tag{$t' = t \cup \{j\}$} .
\end{align*}
Because $|s|=\ell$ and $|t'|=\ell+1$,
$$|t' \setminus s|+|t' \cap s|=\ell+1 =
 |s \setminus t'| + |s \cap t'|+1.$$
We can conclude
\begin{align*}
\sum_{j \not \in s} q_j h_{s,j}
& = \sum_{t': |t'|=\ell+1} \prod_{a \in t'} q_a \sum_{j \in t' , j \not \in s}  \frac{1}{|t' \setminus s|} = g_{\ell+1}(q) .
\end{align*}
This is also true for $\ell=0$.

Finally, define 
$$\Pr[c_{\ell+1} = s \cup \{j\} | c_\ell = s] = \frac{q_j h_{s, j}}{g_{\ell+1}(q)}.$$
We need to show that $c_{\ell+1}$ is properly distributed.
For fixed $s'$ of size $\ell+1$,
\begin{align*}
\sum_{j \in s'} h_{s' \setminus \{j\},j}
& = \sum_{j \in s'} \sum_{t: |t|=\ell, j \not \in t} \frac{1}{|s' \setminus t|} \prod_{a \in t} q_a \\
& =  \sum_{t: |t|=\ell} \prod_{a \in t} q_a \sum_{j \in s' \setminus t} \frac{1}{|s' \setminus t|}   = g_\ell(q) .
\end{align*}
This is also true for $\ell=0$.
So,
\begin{align*}
\Pr[c_{\ell+1}=s']
& = \sum_{j \in  s'} \Pr[c_\ell = s' \setminus \{j\}] 
\frac{q_j \cdot h_{s' \setminus \{j\},j} }{g_{\ell+1}(q)} \\
& = \frac{\prod_{a \in s'} q_a}{g_\ell(q){g_{\ell+1}(q)} } \sum_{j \in  s'}
h_{s' \setminus \{j\},j} = P_{\ell+1}( s') .
\end{align*}
\end{proof}

\begin{proof}[Proof of Lemma~\ref{lem:distChain}]
The set $c_0$ is fixed to be empty.
For $\ell < n$, define $c_{\ell+1}$ from $c_{\ell}$
via Proposition~\ref{prop:Probconst}. 
\end{proof}

\begin{proof}[Proof of Theorem~\ref{thm:LYM}]
Let $C = \{c_0 , c_1,\ldots, c_n\}$ be a random maximal chain
as in Lemma~\ref{lem:distChain}.
Let $L$ be the number of $\ell \in \{0,1,\ldots,n\}$ so that $c_\ell \in A$.
Because $A$ is an antichain, almost surely $L \leq 1$.
On the other hand, 
\begin{align}
\label{eqn:LYMp}
\E L 
%& = \sum_{a \in A} \Pr[a \in C] \\
 & = \sum_\ell \sum_{a \in A:|a|=\ell} \Pr[c_\ell=a] \\
\notag & = \sum_\ell \sum_{a \in A:|a|=\ell} \Pr\big[z=a\big| |z|=\ell \big] \\
\notag & = \sum_\ell \Pr[z \in A||z|=\ell].
\end{align}

\end{proof}

\begin{proof}[Proof of Theorem~\ref{thm:sperner}]
By Theorem~\ref{thm:LYM},
\begin{align*}
\Pr[z \in A] 
& = \sum_{\ell} \Pr[|z|=\ell] \Pr[z \in A|  |z|=\ell] \leq \max_\ell \Pr[|z|=\ell]. 
\end{align*}
\end{proof}

\section{Anti-concentration}
\label{sec:prodMeasure}

Here we prove a general anti-concentration
result for product measures.
The simple proof is inspired by~\cite{rao2018anti}.

\begin{proof}[Proof of Theorem~\ref{thm:anticon}]
Think of $z$ as taking values in $\{\pm 1\}^n$;
this just simplifies the calculations. 
Let $\theta$ be uniformly distributed in $[0,1]$.
For every integer $t$,
\begin{align*}
\Pr\Big[\sum_{j=1}^n z_j =t\Big]
& = \E_\theta \E_z \exp \Big( 2 \pi i \theta \Big(\sum_j z_j-t\Big)\Big) \\
& \leq \E_\theta \Big| \prod_j \E_{z_j} \exp(2 \pi i \theta z_j) \Big| .
\end{align*}
For each $j$,
$\E_{z_j} \exp({2 \pi i \theta z_j})
= p_j \exp({2\pi i \theta})+(1-p_j) \exp(-2 \pi i \theta)$. 
So,
\begin{align*}
|\E_{z_j} \exp(2 \pi i \theta z_j)|^2
& = p_j^2 + (1-p_j)^2 + 2 p_j (1-p_j) \cos(2 \pi \theta) \\ 
%& = 1 + 2p_j^2 -2 p_j + 2 p_j (1-p_j) \cos(2 \pi \theta) \\ 
%& = 1 + 2p_j (p_j-1) + 2 p_j (1-p_j) \cos(2 \pi \theta) \\ 
& = 1 - 2 p_j (1-p_j) (1-\cos(2 \pi \theta)) \\
& = 1 - 4 p_j (1-p_j) \sin^2( \pi \theta) \\
& \leq \exp(-4 p_j (1-p_j) \sin^2( \pi \theta) ).
\end{align*}
It follows that 
\begin{align*}
\Pr\Big[\sum_{j=1}^n z_j =t\Big]
& \leq \E_\theta \exp (-2 \sigma^2_P \sin^2( \pi \theta)  ) \\
& = 2\int_0^{1/2}   \exp (-2 \sigma^2_P \sin^2( \pi \theta)) d \theta \\
& \leq \int_{-\infty}^{\infty}   \exp (-  \sigma^2_P  \theta^2 ) d \theta 
\tag{$\sin(\xi) \geq \xi/2$} \leq \frac{C_1}{\sigma_P}.
\end{align*}

\end{proof}

\section{Antichains of edges}
\label{sec:antiedges}

The following section explains how to bound from above the measure
of {\em oriented} antichains of {\em edges} in general product distributions. 
Oriented antichains of edges are more difficult to handle than
antichains of vertices. 
The arguments in this section are, consequently, more technical 
and less accurate than in the previous sections. 
Nevertheless, the arguments are inspired by
and rely on Section~\ref{sec:antiVer}.

The cube can be oriented so that $u \in \{0,1\}^n$ is its new origin.
This orientation allows to stratify the cube according to 
the Hamming distance from $u$.
%The distance of a vertex $x$ from $u$ is denoted by $|x+u|$,
%which corresponds to addition modulo two.
A $u$-chain of edges is a set 
$C=\{(c_0,c_1),(c_1,c_2),\ldots,(c_{n-1},c_n)\}$
of edges in the cube so that the Hamming distance of $c_\ell$ from 
$u$ is $\ell$, for all $\ell\in \{0,1,\ldots,n\}$.
A $u$-antichain of edges is a set of edges $A$ so that
every $u$-chain of edges intersects $A$ in at most a single edge.

The set of edges $A_h$ that are sliced by a hyperplane $h$ is an oriented
antichain of edges~\cite{O_Neil_1971}.
The origin $u$ is defined by the sign pattern 
of the normal vector $v$ via $u_j = 1$ iff $v_j > 0$.
The set $A_h$ is a $u$-antichain of edges.

\begin{theorem}
\label{thm:baker}
There is a constant $C_2>0$ so that the following holds.
Let $P$ be a non trivial product distribution on $\{0,1\}^n$.
There is a distribution on edges $[x,y]$ in the cube
so that $x \sim P$ with the following property.
Let $u \in \{0,1\}^n$ and let $A$ be a $u$-antichain of edges.
Then,
$$\Pr[ [x,y] \in A] \leq \frac{C_2 n^4}{\sigma^9_P}.$$
\end{theorem}

\begin{remark}
It is important that the distribution on $[x,y]$
does not depend on $u$.
\end{remark}

\begin{proof}
We start by partitioning $[n]$ according to $P$.
For $j \in [n]$, let
$$p_j = \Pr_{z \sim P}[z_j=1].$$
Let 
$$J_* = \{j \in [n] : p_j(1-p_j) \geq p_*\},$$
where
$$p_* = \frac{\sigma_P^2}{2n}.$$
Let $n_* = |J_*|$ and let
$P_*$ be the restriction of $P$ to coordinates in $J_*$.
Because $\sigma_P^2 = \sum_{j\in [n]} p_j(1-p_j)$,
we know that 
\begin{align}
\label{eqn:sigmaP}
\sigma^2_{P_*} = \sum_{j \in J_*} p_j(1-p_j) \geq \frac{\sigma_P^2}{2}.
\end{align}

The distribution on the edge $[x,y]$ is defined as follows.
Let $x \sim P$.
Denote by $N_*(x)$ the set of neighbors of $x$
that differ from $x$ only in coordinates in $J_*$.
Let $y$ be chosen uniformly at random in $N_*(x)$.
The vertices $x,y$ agree in all coordinates not in $J_*$.
For the rest of the proof, we ignore coordinates not in $J_*$,
and think about them as fixed.

The orientation of the cube is determined by $u \in \{0,1\}^n$.
We claim that we can assume without loss of generality that $u=0$.
Why?
Because we can replace $P$ by 
the product distribution $P_u$ on $\{0,1\}^n$
so that $z+u \sim P_u$ where $z \sim P$.
And we can replace the $u$-antichain $A$
by the $\emptyset$-antichain $A+u$,
obtained by replacing each edge $[a,b]$ in $A$
by the edge $[a+u,b+u]$.
The distribution $P_u$ yields exactly the same set $J_*$
as $P$ does. Consequently,
$x+u$ is distributed like $P_u$,
and the distribution of $y+u$ conditioned on $x+u$ remains the same.
In addition, the edge $[x+u,y+u]$ belongs to $A+u$
iff the edge $[x,y]$ belongs to $A$.

Proposition~\ref{prop:Probconst} allows to sample
a vertex at level $\ell+1$ from a vertex at level $\ell$.
By repeatedly applying the proposition, 
we can grow a chain from any starting vertex
either towards $\emptyset$ or in the opposite direction.
A crucial property of this construction is its reversibility.
For example, if $z \sim P_*$ and we grow a chain from $z$
towards $\emptyset$, then we can generate the same distribution
by growing a chain from $\emptyset$ outwards,
and stop at a random time distributed like $|z|$.
This reversibility is instrumental in the proof of Lemma~\ref{lem:C+} below.

To handle the unknown orientation, we generate three random chains of edges
$C,C_+$ and $C_-$:

\begin{enumerate}
\item Let $C$ be a random chain of edges that contains the edge $[x,y]$
as follows.
Use Proposition~\ref{prop:Probconst} with $P_*$
in two opposite directions.
From the minimum between $x,y$ grow a chain towards $\emptyset$.
From the maximum between $x,y$ grow a chain towards $J_*$.
The chain $C$ is obtained by gluing these two chains
via the edge $[x,y]$.

\item The chain $C_+$ is generated as follows.
Let $y_+$ be chosen uniformly at random in $N_*(x)$ conditioned on
$|y_+ \cap J_*|=|x\cap J_*|+1$; when there is no such option,
set $C_+$ to be empty.
As in the construction of $C$,
the chain $C_+$ is obtained by
growing a chain that contains the edge $[x,y_+]$.

\item The chain $C_-$ is generated similarly.
Let $y_-$ be uniformly random in $N_*(x)$ conditioned on
$|y_-\cap J_*|=|x\cap J_*|-1$; when there is no such option,
set $C_-$ to be empty.
As in the construction of $C$ and $C_+$,
the chain $C_-$ is obtained by
growing a chain that contains the edge $[y_-,x]$.

\end{enumerate}

There is a coupling of the three chains $C,C_+$ and $C_-$
so that each chain separately has the correct distribution,
and the edge $[x,y]$ is always equal to one of $[x,y_+]$ or $[y_-,x]$.
First, choose $x\sim P$.
Then choose a neighbor $y_+$ of $x$ 
uniformly at random on the level above $x$,
and choose a neighbor $y_-$ of $x$ 
uniformly at random on the level below $x$.
If one of the two sets of neighbors is empty
then the corresponding chain is empty.
To generate $C_+$ extend the edge $[x,y_+]$
to a chain.
To generate $C_-$ extend the edge $[y_-,x]$
to a chain as well.
Finally, choose $C$ to be one of $C_+$
or $C_-$ with the appropriate chance
so that $y \in \{y_+,y_-\}$ is a uniform neighbor of $x$.

%We can generate the chains in a different way.
%Focus, for example, on $[x,y_+]$ and $C_+$.
%First, grow a chain $\tilde C_+$ all the way
%from $\emptyset$ to $J_*$.
%Second, choose a height $h$ according to $|z|$ for $z \sim P_*$.
%If $h=n_*$ then set $C_+$ to be empty.
%Third, put $[x,y_+]$ at height $h$ with the correct distribution,
%meaning that $x$ is the $h$'th vertex in $\tilde C_+$
%and $y_+$ is a uniformly random neighbor of $x$ above $x$, ignoring $\tilde C_+$.
%Fourth, continue growing the chain after $y_+$,
%still ignoring $\tilde C_+$.
%This generates the same distribution
%by the reversibility property discussed above.

The key idea is to control the distribution of $x$ conditioned on $C_+$.
\begin{lemma}
\label{lem:C+}
There is a universal constant $K>0$ so that the following holds.
Denote by $c_{0},c_1,\ldots,c_{n_*}$
the ${n_*}+1$ vertices that appear in some fixed chain $c_+$ of edges.
Then, for all $\ell  \in \{0,1,\ldots,{n_*}\}$, we have $\Pr[x=c_\ell|C_+=c_+]
 \leq \frac{K n^4}{\sigma^9_P}$.
\end{lemma}

\begin{proof}[Proof of lemma]
For every $\ell$,
\begin{align*}
\Pr[x=c_\ell|C_+=c_+]
& = \frac{\Pr[x=c_\ell] \Pr[C_+=c_+|x=c_\ell]}{\Pr[C_+=c_+]}  .
\end{align*}
Let $Q(s,j) = Q(s,\{j\})$ be the probability to move from $s$ to $s \cup \{j\}$
according to the construction in Proposition~\ref{prop:Probconst};
here $j$ is assumed to be not in $s$.
Let $Q_{j,j+1} = Q(c_j , c_{j+1} \setminus c_j)$.
Similarly, denote by $Q_{a+1,a}$ the probability to move
from $c_{a+1}$ to $c_{a}$.
By construction, for every $j<n_*$,
\begin{align*}
\frac{\Pr[x = c_j]}{\Pr[|x \cap J_*| = j]} Q_{j,j+1} & = 
\Pr[x = c_j||x \cap J_*| = j]
Q_{j,j+1} \\
& = \Pr[x = c_{j+1}||x \cap J_*| = j+1] Q_{j+1,j} \\
& = \frac{\Pr[x = c_{j+1}]}{\Pr[|x \cap J_*| = j+1]} Q_{j+1,j} .
\end{align*}
So,
\begin{align*}
Q_{j+1,j} =
\frac{\Pr[x = c_j]}{\Pr[|x \cap J_*| = j]} \frac{\Pr[|x \cap J_*| = j+1]}{\Pr[x = c_{j+1}]}
Q_{j,j+1}.
\end{align*}
Focus on (for $\ell < n_*$)
\begin{align*}
& \Pr[C_+=c_+|x=c_\ell] \\
& = \frac{1}{n_*-\ell} \prod_{j>\ell} Q_{j,j+1}
\prod_{a< \ell} Q_{a+1,a} \\
& = \frac{1}{n_*-\ell} \prod_{j>\ell} Q_{j,j+1}
\prod_{a< \ell} Q_{a,a+1} 
\frac{\Pr[x = c_a]}{\Pr[|x \cap J_*| = a]} \frac{\Pr[|x \cap J_*| = a+1]}{\Pr[x = c_{a+1}]}
\\
& = \frac{1}{(n_*-\ell)Q_{\ell,\ell+1}} \frac{\Pr[|x \cap J_*| = \ell]}{\Pr[|x \cap J_*| = 0]}
\frac{\Pr[x = c_0]}{\Pr[x = c_\ell]}
\prod_{j} Q_{j,j+1}.
\end{align*}
Hence, for every $\ell \neq \ell'$ that are less than $n_*$,
\begin{align}
\label{eqn:Prell}
\frac{\Pr[x=c_\ell|C_+=c_+]}{\Pr[x=c_{\ell'}|C_+=c_+]} 
& = 
\frac{\Pr[x=c_\ell]\Pr[C_+=c_+|x=c_\ell]}{\Pr[x=c_{\ell'}]\Pr[C_+=c_+|x=c_{\ell'}]} \\ \notag
& = \frac{(n_*-\ell')Q_{\ell',\ell'+1}\Pr[|x \cap J_*|=\ell]}{(n_*-\ell) Q_{\ell,\ell+1}\Pr[|x \cap J_*|=\ell']} .
\end{align}
We need to understand this ratio.

In Proposition~\ref{prop:Probconst}, we used the notation
$q_j = \frac{p_j}{1-p_j}$.
By construction, for every $s$ of size $\ell$ and $j$ not in $s$,
\begin{align*}
Q(s,j) = \frac{q_j}{g_{\ell+1}(q)} \sum_{t: |t|=\ell, j \not \in t} 
\frac{1}{|(s \cup \{j\}) \setminus t|} \prod_{a \in t} q_a .
\end{align*}
Let $j'$ be not in $s \cup \{j\}$.
There is a one-to-one correspondence between 
sets $t$ of size $\ell$ that do not contain $j$,
and sets $t'$ of size $\ell$ that do not contain $j'$;
if $t$ does not contain $j'$ then $t'=t$,
and if $t$ contains $j'$ then $t' = (t \setminus \{j'\}) \cup \{j\}$.
For every $t$ and the corresponding $t'$,
\begin{align*}
\frac{q_j}{|(s \cup \{j\}) \setminus t|} \prod_{a \in t} q_a 
& \leq
\max\Big\{1, \frac{q_{j}}{q_{j'}}\Big\} \frac{q_{j'}}{|(s \cup \{j'\}) \setminus t'|} \prod_{a \in t'} q_a. 
\end{align*}
We can always bound
\begin{align*}
\frac{q_j}{q_{j'}} \leq \frac{\frac{1}{p_j(1-p_j)}}{p_{j'}(1-p_{j'})}
\leq \frac{1}{p_*^2} . 
\end{align*}
It follows that for all $j,j'$ we have $\frac{Q(s,j)}{Q(s,j')}
\leq \frac{1}{p_*^2}$. 
Because $j \mapsto Q(s,j)$ is a probability distribution
on a universe of size $n_*-|s|$, 
we can conclude that for all $j$,
$$p_*^2 \leq (n_*-|s|)Q(s,j) \leq \frac{1}{p_*^2}.$$

It now follows that for all $\ell \neq \ell'$ that are less than $n_*$,
\begin{align*}
p_*^4 \leq \frac{(n_*-\ell')Q_{\ell',\ell'+1}}{(n_*-\ell) Q_{\ell,\ell+1}} 
\leq \frac{1}{p_*^4} .
\end{align*}

Together with~\eqref{eqn:Prell} we get that 
for every $\ell < n$,
\begin{align}
\label{eqn:finalX|c}
\Pr[x=c_\ell|C_+=c_+]
 \leq \frac{2}{p_*^4} \Pr[|x \cap J_*|=\ell].
 \end{align}
Indeed, if for some $\ell_0$ the above does not hold then
for all $\ell<n_*$,
$$\Pr[x=c_{\ell}|C_+=c_+] > 2 \Pr[|x \cap J_*|=\ell]$$
(we can assume that $\frac{C_1}{\sigma_P}< \frac{1}{2}$).
Summing over $\ell < n_*$,
Theorem~\ref{thm:anticon} yields a contradiction:
$$1 > 2(1-\Pr[|x \cap J_*| = n_*]) > 1.$$

Finally,~\eqref{eqn:finalX|c}, the choice of $p_*$, and
another application of Theorem~\ref{thm:anticon} 
complete the proof of the lemma.
\end{proof}
The lemma allows to control the behavior of $[x,y_+]$:
\begin{align*}
\Pr[[x,y_+] \in A]
& = \E_{C_+} \Pr[[x,y_+] \in A| C_+] \\
& \leq \E_{C_+} \Pr[|x| =  h(C_+)|C_+] 
 \leq \frac{K n^4}{\sigma^9_P} ,
\end{align*}
where $h(C_+)$ is the height of the unique edge in 
$C_+ \cap A$ (if there is such an edge).
A similar bound holds for $\Pr[[y_-,x] \in A]$.
By the coupling described above,
$$\Pr[[x,y] \in A] \leq
\Pr[[x,y_+] \in A] + \Pr[[y_-,x] \in A].$$

\end{proof}

\section{Oriented monotone functions}
\label{sec:monotone}
The following section explains how to bound from above the 
measure of oriented antichains of edges of a specific type.
Antichains that are generated by oriented monotone functions.
A function $f : \{0,1\}^n \to \{0,1\}$ is (locally) monotone
if for every $x \in \{0,1\}^n$ and for every $j \in [n]$
so that $x_j = 0$ it holds that $f(x+e_j) \geq f(x)$,
where addition is coordinate-wise modulo two
and $e_j$ is the unit vector with a one in the $j$'th coordinate.
Again, the hypercube can be oriented so that
$u \in \{0,1\}^n$ is its new origin.
A function $f$ is $u$-monotone if the map
$x \mapsto f(x+u)$ is monotone.

The hyperplane $\ip{z}{v}=\mu$ defines the linear threshold function
$f(z) = \mathsf{sign}(\ip{z}{v} - \mu)$.
The set of edges $A$ that are sliced by this hyperplane is not just an oriented antichain. It has additional structure.
The map $f$ is $u$-monotone with $u$ 
so that $u_j = 1$ iff $v_j > 0$.
The set of sliced edges $A$ is deeply related to the sensitivity
of~$f$.
Fourier analysis is a standard tool for controlling sensitivity.
The specific implementation for general product measures
we use appears in the work of Bshouty and Tamon~\cite{bshouty1996fourier}.

The sensitivity of oriented monotone functions, however, is not
our main concern. We need to find distributions
on edges that typically produce unsliced edges.

\begin{theorem}
\label{thm:monotone}
Let $P$ be a non trivial product distribution on $\{0,1\}^n$.
There is a distribution on edges $[x,y]$ 
so that $x \sim P$ with the following property.
Let $u \in \{0,1\}^n$ and let
$f : \{0,1\}^n \to \{0,1\}$ be $u$-monotone. 
Then,
$$\Pr[f(x) \neq f(y)] \leq \frac{1}{\sigma_P}.$$
\end{theorem}

\begin{remark}
It is important that the distribution of $[x,y]$
does not depend on $u$.
\end{remark}

\begin{remark}
Theorem~\ref{thm:monotone} is less general
but quantitively stronger than Theorem~\ref{thm:baker}.
\end{remark}

\begin{proof}
Choose $x \sim P$.
The distribution of $y$ conditioned on $x$ is as follows.
For $j \in [n]$, let $\sigma_j = p_j(1-p_j)$.
For each $j \in [n]$, choose $y = x+e_j$ with probability 
$\frac{\sigma^2_j}{\sigma_P^2}$.
It remains to bound the probability that $f(x) \neq f(y)$.

Let $F$ be the map defined by $F(z) = f(u+z)$.
By assumption, the map $F$ is monotone.
Denote by $Q$ the distribution of $x+u$.
Let $z \sim Q$ and let $w$ be a neighbor of $z$
distributed so that $w = z+e_j$ with probability 
$\frac{\sigma^2_j}{\sigma_P^2}$.
The distribution of $[x+u,y+u]$ is identical to that of $[z,w]$ so that
$$\Pr[f(x) \neq f(y)] = \Pr[F(z) \neq F(w)].$$

We briefly mention the essentials of Fourier analysis
over the hypercube.
For $j \in [n]$, let $q_j = \Pr[x_j+u_j=1]$.
The $n$ random variables $\frac{z_1-q_1}{\sigma_1},\ldots,\frac{z_n-q_n}{\sigma_n}$ are independent 
variables with expectation zero and variance one.
The $2^n$ maps 
$$\chi_S(z) = \prod_{j \in S} \frac{z_j-q_j}{\sigma_j}$$
form an orthonormal basis for the space of functions
from $\{0,1\}^n$ to $\R$ with respect to the inner product
$\E_{z \sim Q} h(z) g(z)$.
Write $F$ as
$$F(z) = \sum_{S \subseteq [n]} \hat{F}(S) \chi_S(z).$$
Because $F$ is Boolean, we know $\sum_S (\hat{f}(S))^2 \leq 1$.

The sensitivity of a monotone function is known
to have a nice Fourier-theoretic representation. 
For $j \in [n]$ and $z \in \{0,1\}^n$,
let $z^{j\to1}$ be the vector that is the same as $z$
except that it always has a one in the $j$'th coordinate,
and let $z^{j \to 0}$ be defined similarly.
Because $F$ is monotone, 
\begin{align*}
\Pr[F(z) \neq F(w)]
& = \E_z \sum_j \frac{\sigma_j^2}{\sigma_P^2}
1_{F(z) \neq F(z+e_j)} \\
& = \sum_j  \frac{\sigma_j}{\sigma_P^2}
\sigma_j \E_z  (F(z^{j \to 1}) - F(z^{j \to 0})) .
\end{align*}
For fixed $j \in [n]$,
denote by $z_{-j}$ the vector $z$ after deleting the $j$'th
coordinate so that
\begin{align*}
\sigma_j \E_z  (F(z^{j \to 1}) - F(z^{j \to 0}))
& = \E_z  \frac{q_j(1-q_j)}{\sigma_j} (F(z^{j \to 1}) - F(z^{j \to 0})) \\
& =  \E_z  q_j \frac{1-q_j}{\sigma_j} F(z^{j \to 1}) 
+ (1-q_j) \frac{-q_j}{\sigma_j} F(z^{j \to 0}) \\
& =  \E_{z_{-j}} \E_{z_j}  \chi_{\{j\}}(z) F(z) \\
& =  \hat{F}(\{j\}) .
\end{align*}
By Cauchy-Schwartz, we can deduce
\begin{align*}
\Pr[F(z) \neq F(w)]
 = \frac{1}{\sigma_P^2} \sum_j \sigma_j \hat{F}(\{j\}) \leq \frac{1}{\sigma_P^2} \sigma_P \sqrt{ \sum_j (\hat{F}(\{j\}))^2} 
 \leq \frac{1}{\sigma_P}.
\end{align*}
\end{proof}

\section{Vectors with many scales}
\label{sec:scales}

In this section we prove a strong anti-concentration bound
for vectors with many scales
(see definition below).
We start with a simple claim.

\begin{claim}
\label{clm:positiveProb}
There is a constant $C_0>1$ so that the following holds.
If $u \in \R^n$ has norm $\|u\|_2 =1$ then
$$\Pr_{x \sim \{\pm 1\}^n} [\tfrac{1}{C_0} \leq |\ip{x}{u}| \leq C_0]
\geq \frac{1}{C_0},$$
where $x\sim \{\pm 1\}^n$ means that $x$ is uniformly distributed
in the cube.
\end{claim}

\begin{proof}
Let $Z = (\ip{x}{u})^2$. So,
$\E Z =1$
and
$$\E Z^2 \leq 6 \|u\|_2^4 + \|u\|_4^4 \leq 7.$$
By the Payley-Zygmond inequality,
\begin{align*}
\Pr[Z \geq C_0^{-2}] \geq 
\frac{(1-C_0^{-2})^2}{7} .
\end{align*}
By Markov's inequality,
$$\Pr[ Z \geq C_0^2 ] \leq \frac{1}{C_0^2}.$$
\end{proof}

\begin{defn}
The vector $v \in \R^n$ has $S$ scales if
$v$ can be partitioned to $S$ vectors
$v^{(1)},v^{(2)},\ldots,v^{(S)}$
so that
$$\|v^{(s)}\|_2 \geq 4C_0^2 \|v^{(s+1)}\|_2$$
for all $s < S$, where $C_0$ is
from Claim~\ref{clm:positiveProb}.
The smallest scale of $v$ is defined to be $\|v^{(S)}\|_2$.
\end{defn}

\begin{lemma}
\label{lem:probBound}
There is a constant $C_3 > 1$ so that the following holds.
If $v \in \R^n$ has $S$ scales and its smallest scale is $\delta \geq 0$,
then for every $a \in \R$ and $b \geq 2$ we have
$$\Pr_{x \sim \{\pm 1\}^n}[|\ip{x}{v} - a| < b \delta] < C_3 \exp(-\tfrac{S}{C_3}
+C_3 \log(b)).$$
\end{lemma}

\begin{proof}
For $s \in [S]$ let $z_s = |\ip{x^{(s)}}{v^{(s)}}|$,
where we partitioned $x$ according to the partition of $v$.
Let $C>0$ be a large enough constant to be determined.
By Chernoff's inequality and Claim~\ref{clm:positiveProb}, with probability at least
$1-C\exp(- S/C)$, there are at least $S/C$
values of $r \in [S]$ so that
\begin{align}
\label{eqn:R1}
\frac{\|v^{(r)}\|_2}{C_0} \leq z_r \leq C_0 \|v^{(r)}\|_2.
\end{align}
Condition on this very likely event.
For every such $r < r'$,
$$z_{r} \geq \frac{\|v^{(r)}\|_2}{C_0}
\geq 4C_0 \|v^{(r')}\|_2 \geq 4 z_{r'}.$$
In particular, there are at least $\tfrac{S}{C}-C\log(b)$ such $r$'s so that
\begin{align}
\label{eqn:R2}
z_r > 4 b \delta.
\end{align}
Denote by $R$ the set of $r \in [S]$ so that both
\eqref{eqn:R1} and~\eqref{eqn:R2} hold.
Let $z = \sum_{s \not \in R} z_s$.
Conditioned on $R$, on the values of $z_r$ for $r \in R$, 
and on $z$,
the event that $|\ip{x}{v} - a| < b \delta$
is the same as
\begin{align}
\label{eqn:Revent}
\Big| z-a+\sum_{r \in R} \eps_r z_r \Big| < b\delta,
\end{align}
where $\eps \sim \{\pm 1\}^R$.
We claim that there is at most a single
$\eps$ so that~\eqref{eqn:Revent} holds.
Indeed, if $\eps \neq \eps'$ both satisfy~\eqref{eqn:Revent}
then
\begin{align*}
2 b \delta 
& > \Big| \sum_{r \in R} (\eps_r-\eps'_r) z_r  \Big| \\
& = \Big| (\eps_{r_0}-\eps'_{r_0}) z_{r_0} + \sum_{r \in R: r > r_0} (\eps_r-\eps'_r) z_r  \Big|,
\end{align*}
where $r_0$ is the minimum index in which $\eps$ and $\eps'$
differ (so $|\eps_{r_0}-\eps'_{r_0}|=2$).
This is a contradiction because 
\begin{align*}
z_{r_0} >  b \delta + \sum_{r \in R : r > r_0} z_r  .
\end{align*}
\end{proof}

\section{Decomposing a matrix}
\label{sec:decom}

The following lemma helps to identify where we need to focus our attention.

\begin{lemma}
\label{lem:VtoV'}
Let $V$ be the matrix whose rows are the skewed $v_1,\ldots,v_k \in \R^n$
with $k \leq n^{0.51}$ and $n$ large enough.
We can re-order the rows and columns of $V$ as follows.
Let $V' = (v'_{ij})$ be the matrix obtained by looking at the first 
$k' \leq k$ rows of $V$,
the first $n' \geq n/2$ columns of $V$,
and renormalizing each row of $V'$ to have $\ell_2$-norm one.
The following hold:
\begin{enumerate}
\item Every row $i > k'$ in $V$ that is not in $V'$ has $S := \lfloor n^{0.001} \rfloor$ scales, and the position of the smallest scale of row $i$ contains the first $n'$ columns.
\item
For every column $j \leq n'$ in $V'$,
\begin{align}
\label{eqn:norm2}
\sum_{i \leq k'} {v'}_{ij}^2 < n^{-0.487}.
\end{align}
In particular,
\begin{align}
\label{eqn:norm1}
\sum_{i \leq k'} |v'_{ij}|
< \sqrt{k n^{-0.487}} \leq n^{0.0115}.
\end{align}
\end{enumerate}

\end{lemma}

\begin{remark}
The value of $k'$ can be zero.
In this case, all rows of $V$ have $S$ scales.
\end{remark}

\begin{proof}
It is convenient to use the following terminology.
The mass of a vector $u$ is $\|u\|_2^2$.
If there is no column $j$ in $V$ with mass
$$\sum_i v_{ij}^2 \geq n^{-0.488},$$
then we are done.
Otherwise, there are such columns.
Start moving them one-by-one to the end of $V$,
and removing them from $V'$ as described below.

When we remove a column from $V'$,
the norm of each row changes, but we do not immediately renormalize it.
Let $\tau >0$ be so that
$$\sqrt{\frac{1-\tau}{\tau}} = 4 C_0^2,$$
where $C_0$ is from Lemma~\ref{lem:distChain}.
If for a given row $i$, after the removal of a column,
the total current mass of the part
that remained in $V'$ drops below $\tau$, then mark a ``drop'' for row $i$,
and renormalize its norm in $V'$ to be one.
When a drop occurs, we get one more scale for $v_i$,
by the choice of $\tau$.
If a row is dropped more than $S$ times, then move it to the end of $V$ and remove it from $V'$. Each removed row has $S$ scales
with the position of the minimum scale containing all columns of $V'$.
This explains the first item.

Let 
$\sigma_t$ be the mass of the column
that was removed at time $t$, with the normalization at time $t$.
So, for all $t$,
$$\sigma_t \geq n^{-0.488}.$$
If there are $T$ iterations, then
$$\sum_t \sigma_t \geq n^{-0.488} T.$$
Let $\sigma_{t,i}$ be the contribution of row $i$ to 
$\sigma_t$; it is zero if the row was removed before time $t$.
Because every row is dropped at most $S$ times, for each row $i$, we have
$\sum_t \sigma_{t,i} < S$.
So,
$$\sum_t \sigma_t < Sk \leq n^{0.511}.$$
It follows that after $T < n^{0.999} \leq \frac{n}{2}$
steps, we get a matrix $V'$ so that the mass of every column $j$ in it is
$$\sum_{i \leq k'} {v'}_{ij}^2 < n^{-0.488}.$$
The rows of this matrix are not yet normalized.
The norm of each row is at least $\sqrt{\tau}$.
The renormalization can not increase the norm of the columns
by more than~$\frac{1}{\sqrt{\tau}}$.
After the final renormalization we get
\begin{align*}
\sum_{i \leq k'} {v'}_{ij}^2 < \frac{n^{-0.488}}{\sqrt{\tau}} < n^{-0.487}.
\end{align*}
\end{proof}

\section{Finding the missing edge}

We are ready to prove the main result.
Assume towards a contradiction that the $k\leq n^{0.51}$ hyperplanes 
defined by the skewed $v_1,\ldots,v_k$ and by $\mu_1,\ldots,\mu_k$ slice all edges of the hypercube $\{\pm 1\}^n$ for large $n$.
Our goal is to locate an edge $[x,y]$ that is not sliced.

Let $V$ be the $k \times n$ matrix whose rows are $v_1,\ldots,v_k$.
From the matrix $V$ we get the matrix $V'$ via Lemma~\ref{lem:VtoV'}.
Write $x = (x',x'')$ where $x'$ is the first $n'$ coordinates of $x$, 
and $x''$ is the last $n-n'$ coordinates.
Write each $v_i$ as $(v'_i,v''_i)$ similarly.
Recall that the first $k'$ rows in $V$ are also in $V'$,
and the rest are not in $V'$ and they have $S \approx n^{0.001}$ scales.
We deal with the last rows first.

\begin{claim}
There is $x'' \in \{\pm 1\}^{n-n'}$ so that 
for every $i > k'$, and for every $x' \in \{\pm 1\}^{n'}$,
$$|\ip{x''}{v''_i} - \mu_i| > |\ip{x'}{v'_i}|+\|v'_i\|_\infty.$$
\end{claim}

\begin{proof}
Start by fixing $i > k'$,
and choose $x''$ uniformly at random.
For simplicity of notation, let
$v=v_i$, let $v' = v'_i$ and let $v'' = v''_i$.
The vector $v$ has $S$ scales,
and $v'$ is part of its smallest scale.
Let $a$ be the vector consisting of the largest $S-1$
scales of~$v$.
The vector of the smallest scale in $v$ has two parts:
$v''$
and the part outside the first $n'$ columns,
which we denote by $b$.
Partition $x''$ to $x_a,x_b$ accordingly.
Let $\delta_a$ be the smallest scale of $a$.
By Cauchy-Schwartz,
$$|\ip{x'}{v'}| +\|v'\|_\infty \leq (\sqrt{n}+1)  \|v'\|_2 \leq 
(\sqrt{n}+1)  \delta_a.$$
By Lemma~\ref{lem:probBound}, conditioned on the value $x_b$,
$$\Pr_{x_a}[|\ip{x''}{v''} - \mu_i| \leq |\ip{x'}{v'}| +\|v'\|_\infty]
\leq C_3 \exp(-\tfrac{S-1}{C_3} +C_3 \log(n)).$$
The same bound holds when we average over $x_b$ as well.

Finally, the union bound completes the proof 
because $S$ is large enough.
\end{proof}

Let $x''$ be as promised by the claim above.
What did we achieve?
No matter how we choose $x'$,
and how we choose a neighbor $y=(y',y'')$ of $x=(x',x'')$,
if $y''=x''$, then
all hyperplanes $v_i,\mu_i$ for $i>k'$ do not slice the edge $[x,y]$.
This is a very strong guarantee.

It remains to deal with rows in $V'$.
The opening move is Bang's lemma.
Consider the $k' \times k'$ matrix
$M = V' {V'}^T$
with
$$\gamma_i = \mu_i - \ip{v''_i}{x''}
\quad \text{and} \quad \theta = n^{-0.0115}.$$
By Lemma~\ref{lem:bang}, 
there is $\eps \in \{\pm 1\}^{k'}$ so that for each $i \leq k'$,
\begin{align}
\label{eqn:number}
\left|\ip{u}{v'_i} - \gamma_i\right| \geq \theta,
\end{align}
where
$$u = \theta {V'}^T \eps.$$
By~\eqref{eqn:norm1},
$$\|u\|_\infty \leq 1.$$
We round $u$ to a vertex of the cube
in two different phases.
The first phase takes us almost all the way to a vertex of the cube.
\begin{claim}
There is $w \in \R^{n'}$ so that the following hold:
\begin{enumerate}
\item $\ip{w}{v'_i} = \ip{u}{v'_i}$ for each $i \leq k'$.
\item $\|w\|_\infty \leq 1$.
\item There is $\tilde k \leq k$ so that $|w_j|<1$ for $j \leq \tilde k$
and $|w_j|=1$ for $j> \tilde k$.
\end{enumerate}

\end{claim}
\begin{proof}
Define a sequence $u^{(0)},u^{(1)},\ldots$ as follows.
Set
$$u^{(0)} = u.$$
Let $a^{(1)}$ be orthogonal to $v'_1,\ldots,v'_{k'}$.
Let 
$$u^{(1)} = u^{(0)} + \alpha_1 a^{(1)}$$
so that
$\|u^{(1)}\|_\infty \leq 1$
and (at least) one of the entries in $u^{(1)}$ is $\pm 1$.
This can be obtained by increasing $\alpha_1$ from $0$ to $\infty$,
and stopping at the correct value.
Without loss of generality assume that $|u^{(1)}_{n'}| = 1$.
Let $a^{(2)}$ be orthogonal to $v'_1,\ldots,v'_{k'},e_{n'}$,
where $e_{n'} = (0,\ldots,0,1)$.
Let 
$$u^{(2)} = u^{(1)} + \alpha_2 a^{(2)}$$
so that
$\|u^{(2)}\|_\infty \leq 1$
and (at least) two of the entries in $u^{(2)}$ are $\pm 1$.
We can keep going for $t = n'-k$ steps and get
the desired $w$.
\end{proof}

If $\tilde k =0$ then we are done,
because by~\eqref{eqn:norm2} we know
that $\|v'_i\|_\infty  < n^{-0.2} \ll \theta$.
This means that every edge $[x,y]$ with $x = (w, x'')$
and $y = (y',x'')$ is not sliced by any hyperplane.
So, we can assume that $\tilde k >0$.

We now enter phase two of the construction. 
This phase uses additional randomness
to find the edge we are looking for.
Let $\delta \in \R^{n'}$ be a random vector distributed as follows.
Its coordinates are independent so that
$\delta_j+w_j$ takes values in $\{\pm 1\}$ and
$$\E \delta_j = 0 .$$
We see that the vertex $x' = \delta+w$ satisfies
for all $i$,
$$\E \ip{x'}{v'_i} = \ip{w}{v'_i} = \ip{u}{v'_i}.$$
Let $\tilde x$ be the first $\tilde k$ coordinates of $x'$.
Denote by $P$ the distribution of $\tilde x$.
Its marginals can be computed as follows
$$w_j = \E[ \delta_j + w_j] = 2  \Pr[\delta_j = 1-w_j] - 1$$
so that
$$\Pr[\delta_j=1-w_j] = \frac{1+w_j}{2}.$$
So, the distribution $P$ is a non trivial product distribution on $\{0,1\}^{\tilde k}$.
It also follows that
\begin{align*}
\E \delta^2_j 
& = \frac{1+w_j}{2} (1-w_j)^2  + \frac{1-w_j}{2}(-1-w_j)^2  \\
%& = (1-w_j)^2 \frac{1+w_j}{2} + (1+w_j)^2 \frac{1-w_j}{2} \\
& = \frac{(1-w_j)(1+w_j)}{2} ( 1-w_j + 1+w_j) \\
& = 1-w^2_j .
\end{align*}
And that
$$\sigma^2_P  =\sum_j \frac{1-w_j^2}{4} \leq k.$$
Sample an edge $[\tilde x,\tilde y]$
in the cube $\{0,1\}^{\tilde k}$
using\footnote{We could use Theorem~\ref{thm:baker} instead.} Theorem~\ref{thm:monotone} for the distribution $P$.
The edge $[\tilde x,\tilde y]$ defines the final edge $[x,y]$.
The rest of the analysis is split between two cases.
For each $i$, let
$$\sigma_i^2 = \sum_{j} (1-w^2_j) {v'_{ij}}^2.$$
First, fix some $i$ so that 
$\sigma^2_i \leq n^{-0.0233}$;
if there are no such $i$'s go to case two.
We use one of Bernstein's inequalities.

\begin{theorem}[Bernstein]
\label{thm:bernstein}
Let $z_1,\ldots,z_\ell$ be independent random variables
with mean zero that are almost surely at most two
in absolute value.
Let $\sigma^2 = \sum_j \E z_j^2$.
Then, for all $t > 0$, we have
$\P\Big[\sum_j z_j \geq t \Big] \leq \exp\big( - \frac{t^2}{2 \sigma^2 + 2t}\big).$
\end{theorem}

By~\eqref{eqn:norm2},
we know $\|v'_i\|_\infty < n^{-0.2}$.
So, by choice of $\theta$,
\begin{align*}
\Pr[|\ip{x'}{v'_i}-\gamma_i| \leq \|v'_i\|_\infty]
\leq \Pr[|\ip{\delta}{v'_i}| > n^{-0.0116}] .
\end{align*}
By Bernstein's inequality,
\begin{align*}
\Pr \left[ \left|\ip{\delta}{v'_i}\right| > n^{-0.0116} \right]
& = \Pr \left[ \tfrac{1}{\sigma_i} \left|\ip{\delta}{v'_i}\right| > \tfrac{1}{\sigma_i} n^{-0.0116} \right] \\
& \leq 
2 \exp\Big( - \frac{\tfrac{1}{\sigma_i^2} n^{-0.0232}}{2 + 2 \tfrac{1}{\sigma_i} n^{-0.0116}}\Big) .
\end{align*}
This means that the probability that the edge $[x,y]$
is sliced by $v_i,\mu_i$ is tiny,
regardless of the choice of $\tilde y$,
so that can do a union bound.

We now move to the second and final case.
By~\eqref{eqn:norm2} again,
$$\sum_i \sigma_i^2 = \sum_{j} (1-w^2_j) \sum_{i \leq k'} {v'}_{ij}^2 <  n^{-0.487} 4\sigma^2_P .$$
The number of $i$'s with $\sigma^2_i > n^{-0.0233}$
is at most 
$$4n^{-0.487+0.0233} \sigma^2_P
= 4n^{-0.4637} \sigma^2_P.$$
%If $4n^{-0.4637} \sigma^2_P < 1$ then we are done.
%So we can assume $2 \sigma_P \geq n^{0.23185}$.
Theorem~\ref{thm:monotone} bounds from above
the probability that $[x,y]$ is sliced by 
a single hyperplane
(see also the discussion before the theorem).
The union bound over all $i$'s with $\sigma^2_i > n^{-0.0233}$
shows that the probability that one of them slices
$[x,y]$ is at most
%$$4n^{-0.4637} \sigma^2_P \cdot \frac{C_2 k^4 }{\sigma^9_P} 
%\leq n^{0.01} n^{-0.4637} n^{2.04}n^{-1.62295} \leq n^{-0.01}.$$
$$4n^{-0.4637} \sigma^2_P \cdot \frac{1}{\sigma_P} 
\leq 4n^{-0.4637} \sqrt{k} \leq n^{-0.1} .$$

We are finally done. 
Let us recall the high-level structure of the argument.
The choice of $x''$ takes care of all hyperplanes not in $V'$
via the strong anti-concentration for vectors with many scales.
The construction of $V'$ takes care of most hyperplanes in $V'$
(the ones with small $\sigma_i$) via Bernstein's inequality.
The few hyperplanes that remain are dealt with
by understanding antichains of edges
in general product measures. 

\bibliographystyle{abbrv}
\bibliography{slicing}

\end{document}